\documentclass[english]{article}
\usepackage[T1]{fontenc}
\usepackage[latin9]{inputenc}
\usepackage{geometry}
\usepackage{color}
\usepackage{babel}
\usepackage{mathtools}
\usepackage{amsmath}
\usepackage{amsthm}
\usepackage{amssymb}
\usepackage[unicode=true,pdfusetitle,
 bookmarks=true,bookmarksnumbered=false,bookmarksopen=false,
 breaklinks=false,pdfborder={0 0 1},backref=false,colorlinks=true]
 {hyperref}

\makeatletter
\numberwithin{equation}{section}
\numberwithin{figure}{section}
\theoremstyle{plain}
\newtheorem{assumption}{\protect\assumptionname}
\theoremstyle{plain}
\newtheorem{thm}{\protect\theoremname}
\theoremstyle{plain}
\newtheorem{lem}{\protect\lemmaname}
\theoremstyle{plain}
\newtheorem{prop}{\protect\propositionname}
\theoremstyle{remark}
\newtheorem{rem}{\protect\remarkname}
\theoremstyle{remark}
\newtheorem*{acknowledgement*}{\protect\acknowledgementname}

\makeatother

\providecommand{\acknowledgementname}{Acknowledgement}
\providecommand{\assumptionname}{Assumption}
\providecommand{\lemmaname}{Lemma}
\providecommand{\propositionname}{Proposition}
\providecommand{\remarkname}{Remark}
\providecommand{\theoremname}{Theorem}

\begin{document}
\title{Dimension-free Wasserstein contraction of nonlinear filters}
\author{Nick Whiteley\\
Institute for Statistical Science, School of Mathematics, University
of Bristol\\
and the Alan Turing Institute}
\maketitle
\begin{abstract}
For a class of partially observed diffusions, conditions are given
for the map from the initial condition of the signal to filtering
distribution to be contractive with respect to Wasserstein distances,
with rate which does not necessarily depend on the dimension of the
state-space. The main assumptions are that the signal has affine drift
and constant diffusion coefficient and that the likelihood functions
are log-concave. Ergodic and nonergodic signals are handled in a single
framework. Examples include linear-Gaussian, stochastic volatility,
neural spike-train and dynamic generalized linear models. For these
examples filter stability can be established without any assumptions
on the observations.
\end{abstract}

\section{Introduction\label{sec:Introduction}}

\subsection{Setting}

Let $(X_{t})_{t\in\mathbb{R}_{+}}$, called the \emph{signal} process,
be the solution of the stochastic differential equation:
\begin{equation}
\mathrm{d}X_{t}=(\alpha+\beta X_{t})\mathrm{d}t+\sigma\mathrm{d}B_{t},\label{eq:SDE_intro}
\end{equation}
where $\alpha\in\mathbb{R}^{p}$ and $\beta$ is a $p\times p$ matrix
of reals, $\sigma\geq0$ is a scalar, and $(B_{t})_{t\in\mathbb{R}_{+}}$
is $p$-dimensional Brownian motion. Let\emph{ observations} $(Y_{k})_{k\in\mathbb{N}_{0}}$
be each valued in a measurable space $(\mathbb{Y},\mathcal{Y})$,
conditionally independent given $(X_{t})_{t\in\mathbb{R}_{+}}$, and
such that the conditional probability that $Y_{k}$ lies in $A\in\mathcal{Y}$
given $(X_{t})_{t\in\mathbb{R}_{+}}$ is of the form $\int_{A}g_{k}(X_{k\Delta},y)\chi(\mathrm{d}y)$,
for a measure $\chi$ on $\mathcal{Y}$, a function $g_{k}:\mathbb{R}^{p}\times\mathbb{Y}\rightarrow(0,\infty)$
and a constant $\Delta>0$.

The \emph{filtering} distributions $\pi_{k}(x,y_{0:k},\cdot)$, $k\in\mathbb{N}_{0}$,
on the Borel sigma algebra $\mathcal{B}(\mathbb{R}^{p})$, associated
with an initial state $x$ and a realized observation sequence $(y_{k})_{k\in\mathbb{N}_{0}}$,
are defined by
\begin{equation}
\pi_{k}(x,y_{0:k},A)\coloneqq\frac{\mathbf{E}_{x}\left[\mathbf{1}_{A}(X_{k\Delta})\prod_{j=0}^{k}g_{j}(X_{j\Delta},y_{j})\right]}{\mathbf{E}_{x}\left[\prod_{j=0}^{k}g_{j}(X_{j\Delta},y_{j})\right]},\quad A\in\mathcal{B}(\mathbb{R}^{p}),\label{eq:filter_defn_k}
\end{equation}
where $\mathbf{E}_{x}$ denotes expectation with respect to the law
of the solution of (\ref{eq:SDE_intro}) with $X_{0}=x$. When $(y_{0},\ldots,y_{k})$
are replaced in (\ref{eq:filter_defn_k}) by the random variables
$(Y_{0},\ldots,Y_{k})$ distributed according to the above prescription
and with true initialization also $X_{0}=x$, then $\pi_{k}(x,y_{0:k},\cdot)$
is a version of the conditional distribution of $X_{k\Delta}$ given
$(Y_{0},\ldots,Y_{k})$. It shall be assumed throughout that whichever
$x$ and $(y_{k})_{k\in\mathbb{N}_{0}}$ we consider, the denominator
in (\ref{eq:filter_defn_k}) is finite for each $k$, which combined
with $g_{k}(x,y)>0$ implies that $\pi_{k}(x,y_{0:k},\cdot)$ is well
defined as a probability measure.

The filtering problem \textendash{} computing or approximating the
distributions (\ref{eq:filter_defn_k}) \textendash{} appears across
Bayesian statistics, machine learning and signal processing \cite{harrison1999bayesian,murphy2012machine,douc2014nonlinear}
and a broard literature on its mathematical analysis has developed
\cite{crisan2011oxford}. The question of under what conditions the
filtering distributions are stable with respect to their initial condition
has a rich history and has been addressed using a wide variety of
techniques, an overview of the field is given in \cite[Chap. 4]{crisan2011oxford}.

Relatively recent results from \cite{kleptsyna2008discrete,Douc2009,douc2010forgetting,gerber2017stability}
are applicable to the model class described above, or some instances
thereof, under appropriate technical conditions. They establish quantitative
bounds on the total variation distance, or a weighted version thereof
in \cite{gerber2017stability}, between differently initialized filtering
distributions and obtain rate estimates which depend on constants
associated with minorization-type conditions for the signal process.
However such constants, and therefore the rate estimates obtained
from them, typically degrade with the dimension of the state-space.
The emphasis of the present work is on identifying techniques and
assumptions which allow this issue to be overcome.

Also recently, infinite-dimensional filtering has been treated in
\cite{tong2014conditional}, where stability results are obtained
involving weak convergence and the notion of local ergodicity, which
pertains to the mixing properties of non-Markovian, finite-dimensional
components of an infinite dimensional signal process, conditional
on the observations. The results hold under very mild conditions which
cannot be expected to yield a particular rate of convergence. As part
of a study of particle filters for signals with certain spatio-temporal
mixing properties, \cite{rebeschini2015can} uses the Dobrushin comparison
theorem to obtain quantitative filter stability results with respect
to local variation norms, which do not degrade with dimension. 

\subsection{Outline of the approach\label{subsec:Outline}}

The approach taken here does not rely on spatial structure of the
model, but is instead connected with contraction properties of gradient
flows and convexity, and influenced by analyses of Markov processes
using abstract ideas of curvature and underlying links to functional
inequalities \cite{bakry2013analysis,catt2014}. The proofs ultimately
rely on a quite simple coupling technique and the pathwise stability
properties of diffusions whose drifts involve the gradients of certain
convex potentials. This convexity arises from a combination of two
features of the model we consider: firstly a log-concavity-preservation
characteristic of the signal model (\ref{eq:SDE_intro}), and secondly
log-concavity of the likelihood functions $x\mapsto g_{k}(x,y)$ (precise
assumptions are stated later).

Regarding the first feature, it is known that the transition kernels
$(P_{t})_{t\in\mathbb{R}_{+}}$ associated with (\ref{eq:SDE_intro})
preserve log-concavity, meaning that for any log-concave function
$f$ and $t>0$, $P_{t}f$ is log-concave, see for example \cite{kolesnikov2001diffusion}.
If for each $k$ and $y$ the likelihood function $x\mapsto g_{k}(x,y)$
is log-concave, then the Markov property of $(X_{t})_{t\in\mathbb{R}_{+}}$
and the fact that a pointwise product of log-concave functions is
log-concave imply that the function $x\mapsto\mathbf{E}_{x}\left[\prod_{i=j}^{k}g_{i}(X_{(i-j)\Delta},y_{i})\right]$
is log-concave for any $y_{j},\ldots,y_{k}$. Functions of this form
play an important role in filter stability because they provide the
re-weighting of transition probabilities which corresponds to conditioning
on observations, and this is where the convex potentials alluded to
earlier arise.

It is important to note that log-concavity of $x\mapsto\mathbf{E}_{x}\left[\prod_{i=j}^{k}g_{i}(X_{(i-j)\Delta},y_{i})\right]$
cannot be expected in much greater generality. It was established
in \cite{kolesnikov2001diffusion} that among all diffusions of the
form:
\[
\mathrm{d}X_{t}=b(X_{t})\mathrm{d}t+\sigma(X_{t})\mathrm{d}B_{t},
\]
with $b(\cdot)$, $\sigma(\cdot)$ satisfying some mild regularity
conditions, it is only in the case that $b(\cdot)$ is affine and
$\sigma(\cdot)$ is a constant that $P_{t}f$ is log-concave for all
log-concave $f$. This motivates our focus on signal processes of
the form (\ref{eq:SDE_intro}).

Having emphasized the central role of convexity in the present work,
let us finally note the results presented here complement those of
\cite{stannat}, who studied filter stability for a class of diffusions
which are linearly observed in continuous time:
\begin{align}
\mathrm{d}X_{t} & =\beta X_{t}\mathrm{d}t+CC^{T}\nabla\log\phi(X_{t})\mathrm{d}t+C\mathrm{d}B_{t},\label{eq:stannat1}\\
\mathrm{d}Y_{t} & =GX_{t}\mathrm{d}t+\Gamma\mathrm{d}W_{t},\label{eq:stannat2}
\end{align}
where $\beta,C,G,\Gamma$ are matrices of appropriate size, $W_{t}$
is $p$-dimensional Brownian motion and $R\coloneqq\Gamma\Gamma^{T}$
is invertible. Stability is proved in \cite{stannat} via certain
diffusion contraction estimates with respect to Lipschitz norms, under
the condition that 
\[
x\mapsto V(x)\coloneqq\left\langle \beta x,\nabla\log\phi(x)\right\rangle +\frac{1}{2}\frac{\mathrm{tr}(Q\phi^{\prime\prime}(x))}{\phi(x)}+\frac{1}{2}\left\langle R^{-1}Gx,Gx\right\rangle 
\]
is uniformly strictly convex. A rate of convergence in total variation
distance is obtained in terms of the spectrum of the solution of a
particular matrix Riccati equation. The setup (\ref{eq:stannat1})-(\ref{eq:stannat2})
is a counterpart to the one considered in the present paper: in (\ref{eq:stannat1})-(\ref{eq:stannat2})
the linearity is in the observation model, where as in (\ref{eq:SDE_intro})
the linearity is in the signal and our discrete-time observations
$(Y_{k})_{k\in\mathbb{N}_{0}}$ may be related to the signal in a
nonlinear way.

\subsection{Notation and conventions}

The Euclidean norm and inner-product on $\mathbb{R}^{p}$ are denoted
$\|\cdot\|$ and $\left\langle \cdot,\cdot\right\rangle $. A function
$f:\mathbb{R}^{p}\to(0,\infty)$ is called log-concave if 
\[
\log f(cu+(1-c)v)\geq c\log f(u)+(1-c)\log f(v),\qquad\forall\,u,v\in\mathbb{R}^{p},\,c\in[0,1],
\]
and strongly log-concave if there exists a log-concave function $\tilde{f}$
and a constant $\lambda_{f}\in(0,\infty)$ such that $f(u)=\exp(-\frac{\lambda_{f}}{2}\|u\|^{2})\tilde{f}(u)$.
For a measure $\mu$, function $f$ and integral kernel $K$, we shall
write $\mu f=\int f(u)\mu(\mathrm{d}u)$, $\mu K(\cdot)=\int\mu(\mathrm{d}u)K(u,\cdot)$,
$Kf(u)=\int f(v)K(u,\mathrm{d}v)$. For a nonnegative function $f$,
$\mu\cdot f$ denotes the measure $\mu(\mathrm{d}u)f(u)$. The gradient
and Laplace operators with respect to $x$ are denoted $\nabla_{x}$
and $\nabla_{x}^{2}$ . The indicator function on a set $A$ is denoted
$\mathbf{1}_{A}$. The class of real-valued and twice continuously
differentiable functions on $\mathbb{R}^{p}$ is denoted $C^{2}$. 

The order-$q$ Wasserstein distance between probability measures on
$\mathcal{B}(\mathbb{R}^{p})$ is:
\[
W_{q}(\mu,\nu)\coloneqq\left(\inf_{\gamma\in\Gamma(\mu,\nu)}\int_{\mathbb{R}^{p}\times\mathbb{R}^{p}}\|u-v\|^{q}\gamma(\mathrm{d}u,\mathrm{d}v)\right)^{1/q},
\]
where $\Gamma(\mu,\nu)$ is the set of all couplings of $\mu$ and
$\nu$.

\section{Wasserstein distance between filtering distributions initialized
at points}

\subsection{Main result}
\begin{assumption}
\label{assu:g} For each $k\in\mathbb{N}_{0}$ and $y\in\mathbb{Y}$,
$x\mapsto g_{k}(x,y)$ is strictly positive, a member of $C^{2}$,
and there exists $\lambda_{g}(k,y)\in[0,\infty)$ and a log-concave
function $\tilde{g}_{k}(\cdot,y):\mathbb{R}^{p}\rightarrow(0,\infty)$
such that $g_{k}(x,y)=\exp\left[-\frac{\lambda_{g}(k,y)}{2}\|x\|^{2}\right]\tilde{g}_{k}(x,y)$.
\end{assumption}
\begin{thm}
\label{thm:singular}If assumption \ref{assu:g} holds, then for any
$q\geq1$, $k\geq1$ and $y_{0},\ldots,y_{k}\in\mathbb{Y}$,
\begin{equation}
W_{q}(\pi_{k}(x,y_{0:k},\cdot),\pi_{k}(x^{\prime},y_{0:k},\cdot))\leq\exp\left[-\sum_{j=1}^{k}\int_{0}^{\Delta}\lambda(j,y_{j},t)\mathrm{d}t\right]\|x-x^{\prime}\|,\qquad\forall x,x^{\prime}\in\mathbb{R}^{p},\label{eq:thm_1_bound}
\end{equation}
where 
\begin{equation}
\lambda(j,y,t)\coloneqq\lambda_{\mathrm{sig}}+\frac{\sigma^{2}\lambda_{g}(j,y)\lambda_{\beta}^{\mathrm{min}}(\Delta-t)}{1+\sigma^{2}\lambda_{g}(j,y)\int_{t}^{\Delta}\lambda_{\beta}^{\mathrm{max}}(\Delta-s)\mathrm{d}s},\label{eq:lambda_thm1}
\end{equation}
$\lambda_{\mathrm{sig}}\in\mathbb{R}$ is the minimum eigenvalue of
$-(\beta+\beta^{T})/2$ and $\lambda_{\beta}^{\mathrm{min}}(t),\lambda_{\beta}^{\mathrm{max}}(t)\in(0,\infty)$
are respectively the minimum and maximum eigenvalues of $e^{\beta t}(e^{\beta t})^{T}$.
\end{thm}

\subsection{Proof of theorem \ref{thm:singular}\label{subsec:Proof-of-theorem}}

Let $(y_{k})_{k\in\mathbb{N}_{0}}$ be an arbitrary sequence in $\mathbb{Y}$.
This sequence will remain fixed throughout sections \ref{subsec:Proof-of-theorem}-\ref{subsec:Quantifying-log-concavity-of}.
To avoid cumbersome formulae, the dependence of some quantities on
this sequence $(y_{k})_{k\in\mathbb{N}_{0}}$ will not be shown in
the notation, but in some places will be commented on in the text
for avoidance of doubt.

Fix $k\geq0$ and define 
\begin{align}
 & \varphi_{k,k}(x)\coloneqq g_{k}(x,y_{k}),\label{eq:phi_k_defn}\\
 & \varphi_{j,k}(x)\coloneqq g_{j}(x,y_{j})P_{\Delta}\varphi_{j+1,k}(x),\quad0\leq j<k,\label{eq:phi_j_defn}\\
 & R_{j,k}(x,A)\coloneqq\frac{\int_{A}P_{\Delta}(x,\mathrm{d}x^{\prime})\varphi_{j,k}(x^{\prime})}{P_{\Delta}\varphi_{j,k}(x)},\quad1\leq j\leq k.\nonumber 
\end{align}
The dependence of $\varphi_{j,k}$ and $R_{j,k}$ on $y_{j},\ldots,y_{k}$
is not shown in the notation. Here the presentation is heavily influenced
by the semigroup formulation of \cite{del2004feynman}.

We will need the following preliminary results concerning log-concave
functions.
\begin{lem}
\label{lem:strong_log_concavity}For any given $f:\mathbb{R}^{p}\to(0,\infty)$
which is a member of $C^{1}$ and $\lambda_{f}\geq0$, conditions
1)-3) are equivalent:

1) There exists a log-concave function $\tilde{f}$ such that $f(u)=\exp\left(-\frac{\lambda_{f}}{2}\|u\|^{2}\right)\tilde{f}(u)$,
$\forall u\in\mathbb{R}^{p}$.

2) $\log f(u)\leq\log f(v)+\left\langle \nabla\log f(v),u-v\right\rangle -\frac{\lambda_{f}}{2}\|u-v\|^{2},$
$\forall u,v\in\mathbb{R}^{p}$.

3) $\left\langle \nabla\log f(u)-\nabla\log f(v),u-v\right\rangle \leq-\lambda_{f}\|u-v\|^{2},$
$\forall u,v\in\mathbb{R}^{p}$.
\end{lem}
\begin{proof}
These equivalences are immediate consequences of elementary properties
of strongly convex $C^{1}$ functions, see for example \cite[Sec 2.1.3]{nesterov2018lectures}.
\end{proof}
\begin{lem}
\label{lem:preservation_basic}For every log-concave $f$ and $t>0$,
$P_{t}f$ is log-concave.
\end{lem}
\begin{proof}
\cite[proof of Prop. 1.3]{kolesnikov2001diffusion}
\end{proof}
\begin{lem}
\label{lem:varphi_log_concave}We have
\begin{equation}
\pi_{k}(x,y_{0:k},A)=R_{1,k}R_{2,k}\cdots R_{k,k}(x,A).\label{eq:pi_theta_R}
\end{equation}
If assumption \ref{assu:g} holds, then for each $j,k$ such that
$0\leq j\leq k$, there exists a log-concave function $x\mapsto\tilde{\varphi}_{j,k}(x)$,
depending on $y_{j},\ldots,y_{k}$, such that:
\begin{equation}
\varphi_{j,k}(x)=\exp\left[-\frac{\lambda_{g}(j,y_{j})}{2}\|x\|^{2}\right]\tilde{\varphi}_{j,k}(x).\label{eq:varphi_strongly_log_concave}
\end{equation}
\end{lem}
\begin{proof}
The expression for $\pi_{k}(x,y_{0:k},A)$ follows from (\ref{eq:filter_defn_k})
and the Markov property of the signal, this key identify can be traced
back to \cite{del1999stability}. The second claim is established
using assumption \ref{assu:g}, repeated application to (\ref{eq:phi_k_defn})\textendash (\ref{eq:phi_j_defn})
of lemma \ref{lem:preservation_basic} and the fact that the pointwise
product of log-concave functions is log-concave.
\end{proof}
The main steps in the proof of theorem \ref{thm:singular} from hereon
are:
\begin{enumerate}
\item Lemma \ref{lem:transformation} in section \ref{subsec:A-space-time--transform}
establishes that each Markov kernel $R_{j,k}$ can be interpreted
in terms of the transition probabilities of an $h$-transform of the
signal process (\ref{subsec:A-space-time--transform}), where $h$
is function which depends on $y_{j},\ldots,y_{k}$ via $\varphi_{j,k}$.
This transformation amounts to the addition of an extra ``drift''
term to the extended space-time generator (defined below) associated
with the signal, where the extra term depends on $h$.
\item Proposition \ref{prop:wasser_R} in section \ref{subsec:A-space-time--transform}
bounds the Wasserstein distance between $R_{j,k}(x,\cdot)$ and $R_{j,k}(x^{\prime},\cdot)$
using a synchronous coupling of these $h$-transformed diffusions,
assuming log-concavity of the $h$-function in its spatial argument.
Specifically, in the proof of proposition \ref{lem:transformation}
the Wasserstein distance is bounded in terms of the Euclidean distance
between the paths of the coupled diffusions, which is in turn controlled
by $\lambda_{\mathrm{sig}}$ and the strength of the log-concavity
of the $h$-function. Roughly speaking, the stronger this log-concavity
is, the stronger the Wasserstein contraction of $R_{j,k}(x,\cdot)$
is.
\item Proposition \ref{prop:quantitative log-concave} in section \ref{subsec:Quantifying-log-concavity-of}
establishes that the $h$-function is indeed log-concave in its spatial
argument, and quantifies the strength of its log-concavity. In the
proof of proposition \ref{prop:quantitative log-concave} this log-concavity
is inherited from that of $\varphi_{j,k}$ as per (\ref{eq:varphi_strongly_log_concave}),
and it is from here that the constants $\lambda_{g}(k,y_{k})$ from
assumption \ref{assu:g} appear in the log-concavity of $h$ and hence
ultimately in the bounds of proposition \ref{prop:wasser_R}.
\item Finally, the bound on $W_{q}(\pi_{k}(x,y_{0:k},\cdot),\pi_{k}(x^{\prime},y_{0:k},\cdot))$
given in theorem \ref{thm:singular} is an immediate consequence of
proposition \ref{prop:wasser_R} combined with (\ref{eq:pi_theta_R}).
\end{enumerate}

\subsection{A space-time $h$-transform of the signal process\label{subsec:A-space-time--transform}}

Let $C([0,\Delta],\mathbb{R}^{p}\times[0,\Delta])$ be the space of
$\mathbb{R}^{p}\times[0,\Delta]$-valued, continuous functions on
$[0,\Delta]$ endowed with the supremum norm. Let $(X_{t},t)_{t\in[0,\Delta]}$
be the associated space-time coordinate process and let $\mathcal{F}=(\mathcal{F}_{t})_{t\in[0,\Delta]}$
be the filtration it generates. The extended generator (in the sense
of \cite[p. 285]{revuz2013continuous}) of the space-time process
on $C([0,\Delta],\mathbb{R}^{p}\times[0,\Delta])$ under the law associated
with (\ref{eq:SDE_intro}) and acting on functions $f$ on $\mathbb{R}^{p}\times\mathbb{R}_{+}$
is:

\[
Lf(x,t)\coloneqq\frac{\partial}{\partial t}f(x,t)+(\alpha+\beta x)^{T}\nabla_{x}f(x,t)+\frac{\sigma^{2}}{2}\nabla_{x}^{2}f(x,t).
\]

\begin{lem}
\label{lem:transformation}Let assumption \ref{assu:g} hold, fix
any $j,k$ such that $1\leq j\leq k$ and define
\begin{equation}
h(x,t)\coloneqq P_{\Delta-t}\varphi_{j,k}(x),\label{eq:h_defn}
\end{equation}
where the dependence of $h$ on $j,k$ and $y_{j},\ldots,y_{k}$ is
not shown in the notation. There exists a probability kernel $\mathbf{P}^{h}:\mathbb{R}^{p}\times\mathcal{F}_{\Delta}\to[0,1]$
such that for any $x_{0}\in\mathbb{R}^{p}$ and $A\in\mathcal{B}(\mathbb{R}^{p})$,
$R_{j,k}(x_{0},A)=\mathbf{P}^{h}(x_{0},\{X_{\Delta}\in A\})$, and
under $\mathbf{P}^{h}(x_{0},\cdot)$ the extended generator of the
space-time process $(X_{t},t)_{t\in[0,\Delta]}$ on $C([0,\Delta],\mathbb{R}^{p}\times[0,\Delta])$
is:
\begin{equation}
L^{h}f(x,t)\coloneqq Lf(x,t)+\sigma^{2}\nabla_{x}\log h(x,t)^{T}\nabla_{x}f(x,t).\label{eq:L^h}
\end{equation}
\end{lem}
\begin{proof}
Let $\mathbf{P}:\mathbb{R}^{p}\times\mathcal{F}_{\Delta}\to[0,1]$
be a probability kernel such that $\mathbf{P}(x_{0},\cdot)$ is the
law of the space-time process associated with (\ref{eq:SDE_intro})
on the time horizon $[0,\Delta]$ initialized from the point $(x_{0},0)$.

Note the following three properties of $x\mapsto\varphi_{j,k}(x)$:
i) Under assumption \ref{assu:g}, for all $k\geq0$, $x\mapsto g_{k}(x,y_{k})$
is strictly positive and therefore so is $x\mapsto\varphi_{j,k}(x)$
for all $j\leq k$. ii) Under assumption \ref{assu:g} for all $k\geq0$,
$x\mapsto g_{k}(x,y_{k})$ is a member of $C^{2}$, and combined
with (\ref{eq:phi_k_defn})-(\ref{eq:phi_j_defn}) and the fact that
the solution of (\ref{eq:SDE_intro}) satisfies $X_{(k+1)\Delta}=a+BX_{k\Delta}+\sigma\xi_{k+1}$
where $\xi_{k+1}=e^{\Delta\beta}\int_{k\Delta}^{(k+1)\Delta}e^{-(t-k\Delta)\beta}\mathrm{d}B_{t}$
is a Gaussian random variable and $a=e^{\Delta\beta}\int_{0}^{\Delta}e^{-t\beta}\alpha\mathrm{d}t$,
$B=e^{\Delta\beta}$ , this implies $x\mapsto\varphi_{j,k}(x)$ is
a member of $C^{2}$ for all $j\leq k$. iii) By (\ref{eq:varphi_strongly_log_concave})
in lemma \ref{lem:varphi_log_concave} and the equivalence between
conditions 1) and 2) in lemma \ref{lem:strong_log_concavity} with
$f$ there taken to be $\varphi_{j,k}$, we have $\log\varphi_{j,k}(x)\leq\log\varphi_{j,k}(0)+\nabla_{x}\log\varphi_{j,k}(0)^{T}x-\frac{\lambda_{g}(j,y_{j})}{2}\|x\|^{2}$,
hence $\varphi_{j,k}(x)$ grows no faster than $e^{c\|x\|}$ as $\|x\|\to\infty$
where $c=\|\nabla_{x}\log\varphi_{j,k}(0)\|$.

In the remainder of the proof of the lemma, $j,k$ are fixed as in
the statement lemma, and the dependence of various quantities on $j,k$
and $y_{j},\ldots,y_{k}$ is not shown in the notation. Appealing
to the properties of $x\mapsto\varphi_{j,k}(x)$ which have just been
stated, $x\mapsto h(x,t)$ is strictly positive, a member of $C^{2}$,
and log-concave by lemma \ref{lem:varphi_log_concave} and lemma \ref{lem:preservation_basic}.
 With:
\[
D_{t}\coloneqq\frac{h(X_{t},t)}{h(x_{0},0)},
\]
$(D_{t})_{t\in[0,\Delta]}$ is a $(\mathcal{F}_{t},\mathbf{P}(x_{0},\cdot))$-continuous
martingale, and the expected value of $D_{t}$ under $\mathbf{P}(x_{0},\cdot)$
is $1$. Now define the probability kernel $\mathbf{P}^{h}(x,\cdot)\coloneqq D_{\Delta}\cdot\mathbf{P}(x,\cdot)$.
Note that $\mathbf{P}^{h}$ depends on $j,k$ and $y_{j},\ldots,y_{k}$
via $h$. Moreover under $\mathbf{P}^{h}(x_{0},\cdot)$, $(X_{t})_{t\in[0,\Delta]}$
is an inhomogeneous Markov process with transition probabilities:
\[
P_{s,t}^{h}(x,\mathrm{d}x^{\prime})\coloneqq\frac{P_{t-s}(x,\mathrm{d}x^{\prime})h(x^{\prime},t)}{h(x,s)},
\]
and $R_{j,k}(x,A)=P_{0,\Delta}^{h}(x,A)=\mathbf{P}^{h}(x,\{X_{\Delta}\in A\})$.
By \cite[Prop. 3.9, p.357]{revuz2013continuous}, the extended generator
of the space-time process under $\mathbf{P}^{h}(x_{0},\cdot)$ is
$h^{-1}L(hf)$. Using the fact that $\int P_{s}(x,\mathrm{d}x^{\prime})h(x^{\prime},s+t)=h(x,t)$
we have $L(h)=0$, and combining this observation with elementary
differential calculus manipulations it can be checked that $h^{-1}L(hf)$
is equal to the right hand side of (\ref{eq:L^h}). 
\end{proof}
Before stating the following proposition, we emphasize once again
that $(y_{k})_{k\in\mathbb{N}_{0}}$ are fixed.
\begin{prop}
\label{prop:wasser_R}Fix any $j,k$ such that $1\leq j\leq k$. If
there exists a continuous function $\lambda_{h}:[0,\Delta]\to[0,\infty)$
and a function $\tilde{h}:\mathbb{R}^{p}\times[0,\Delta]\to(0,\infty)$
such that for each $t$, $x\mapsto\tilde{h}(x,t)$ is log-concave
and $h$ as in lemma \ref{lem:transformation} satisfies $h(x,t)=\exp\left[-\frac{\lambda_{h}(t)}{2}\|x\|^{2}\right]\tilde{h}(x,t)$,
then for any $q\geq1$,
\[
W_{q}(R_{j,k}(x,\cdot),R_{j,k}(x^{\prime},\cdot))\leq\exp\left[-\lambda_{\mathrm{sig}}\Delta-\sigma^{2}\int_{0}^{\Delta}\lambda_{h}(t)\mathrm{d}t\right]\|x-x^{\prime}\|.
\]
\end{prop}
\begin{proof}
Consider the synchronous coupling:
\begin{align*}
 & X_{t}=x_{0}+\int_{0}^{t}\alpha+\beta X_{s}+\sigma^{2}\nabla_{x}\log h(X_{s},s)\mathrm{d}s+\sigma B_{t},\\
 & X_{t}^{\prime}=x_{0}^{\prime}+\int_{0}^{t}\alpha+\beta X_{s}^{\prime}+\sigma^{2}\nabla_{x}\log h(X_{s}^{\prime},s)\mathrm{d}s+\sigma B_{t}.
\end{align*}
By Ito's formula, for any continuous function $\zeta:[0,\Delta]\to\mathbb{R}$.
\begin{align}
 & \|X_{t}-X_{t}^{\prime}\|^{2}e^{2\int_{0}^{t}\zeta(s)\mathrm{ds}}\nonumber \\
 & \quad=\|x_{0}-x_{0}^{\prime}\|^{2}+2\int_{0}^{t}\left(\zeta(s)\|X_{s}-X_{s}^{\prime}\|^{2}+(X_{s}-X_{s}^{\prime})^{T}\beta(X_{s}-X_{s}^{\prime})\right)e^{2\int_{0}^{s}\zeta(u)\mathrm{d}u}\mathrm{d}s\nonumber \\
 & \quad+2\int_{0}^{t}\sigma^{2}(\nabla_{x}\log h(X_{s},s)-\nabla_{x}\log h(X_{s}^{\prime},s))^{T}(X_{s}-X_{s}^{\prime})e^{2\int_{0}^{s}\zeta(u)\mathrm{d}u}\mathrm{d}s.\label{eq:ito}
\end{align}
Now set $\zeta(s)=\lambda_{\mathrm{sig}}+\sigma^{2}\lambda_{h}(s)$.
For any skew-symmetric matrix, say $A$, and any $u\in\mathbb{R}^{p}$,
$u^{T}Au=(Au)^{T}u=u^{T}A^{T}u=-u^{T}Au$, hence $u^{T}Au=0$, so
\begin{equation}
u^{T}\beta u=\frac{1}{2}u^{T}(\beta+\beta^{T})u\leq-\lambda_{\mathrm{sig}}\|u\|^{2},\qquad\forall u\in\mathbb{R}^{p}.\label{eq:beta_Symm}
\end{equation}
The assumption of the lemma on $h$ combined with lemma \ref{lem:strong_log_concavity}
implies 
\begin{equation}
(\nabla_{x}\log h(x,s)-\nabla_{x}\log h(x^{\prime},s))^{T}(x-x^{\prime})\leq-\lambda_{h}(s)\|x-x^{\prime}\|^{2},\qquad x,x^{\prime}\in\mathbb{R}^{p}.\label{eq:h_slc}
\end{equation}
Applying (\ref{eq:beta_Symm}) and (\ref{eq:h_slc}) to (\ref{eq:ito})
gives:
\[
\|X_{\Delta}-X_{\Delta}^{\prime}\|\leq\exp\left(-\int_{0}^{\Delta}\lambda_{\mathrm{sig}}+\sigma^{2}\lambda_{h}(t)\mathrm{d}t\right)\|x_{0}-x_{0}^{\prime}\|.
\]
The proof is completed by taking expectations and applying lemma \ref{lem:transformation}.
\end{proof}

\subsection{Quantifying log-concavity of $x\protect\mapsto h(x,t)$\label{subsec:Quantifying-log-concavity-of}}

The main result of this section is proposition \ref{prop:quantitative log-concave},
which complements lemma \ref{lem:preservation_basic} by quantifying
the influence on the log-concavity of $x\mapsto h(x,t)$ of the parameters
of the signal process and the log-concavity of the likelihood functions,
and provides verification of the hypotheses of proposition \ref{prop:wasser_R}.
\begin{prop}
\label{prop:quantitative log-concave}Let assumption \ref{assu:g}
hold, fix $j,k$ such that $1\leq j\leq k$ and let $h$ be as in
lemma \ref{lem:transformation}. Then there exists a function $\tilde{h}:\mathbb{R}^{p}\times[0,\Delta]\to(0,\infty)$
such that $x\mapsto\tilde{h}(x,t)$ is log-concave and 
\[
h(x,t)=\exp\left[-\frac{\lambda_{h}(t)}{2}\|x\|^{2}\right]\tilde{h}(x,t),
\]
where 
\[
\lambda_{h}(t)\coloneqq\frac{\lambda_{g}(j,y_{j})\lambda_{\beta}^{\mathrm{min}}(\Delta-t)}{1+\sigma^{2}\lambda_{g}(j,y_{j})\int_{t}^{\Delta}\lambda_{\beta}^{\mathrm{max}}(\Delta-s)\mathrm{d}s},
\]
and $\lambda_{\beta}^{\mathrm{min}}(t),\lambda_{\beta}^{\mathrm{max}}(t)$
are respectively the minimum and maximum eigenvalues of $e^{\beta t}(e^{\beta t})^{T}$.
\end{prop}
We shall make use of the following well-known lemma \cite[Thm. 6]{prekopa1973logarithmic}.
\begin{lem}
\label{lem:Prekopa}For every function $(u,v)\mapsto f(u,v)$ on $\mathbb{R}^{p}\times\mathbb{R}^{q}$
which is log-concave in $(u,v)$, the integral $\int f(u,v)\mathrm{d}v$
is a log-concave function of $u$.
\end{lem}
Lemma \ref{lem:matrix identities} and lemma \ref{lem:integrang_log_concave}
are technical results used in the proof of proposition \ref{prop:quantitative log-concave}.
\begin{lem}
\label{lem:matrix identities}Let $F,S$ be real, square, symmetric
matrices such that $F+S$ is invertible. Then 
\[
v^{T}Fv+(u-v)^{T}S(u-v)=u^{T}Cu+z^{T}(F+S)z
\]
where $C\coloneqq F(F+S)^{-1}S$ and $z\coloneqq v-(F+S)^{-1}Su$. 

\end{lem}
\begin{proof}
We have using the assumed symmetry of $F$ and $S$, 
\begin{align*}
z^{T}(F+S)z & =v^{T}(F+S)v-2u^{T}Sv+u^{T}S(F+S)^{-1}Su.
\end{align*}
Therefore
\begin{align*}
u^{T}Cu+z^{T}(F+S)z & =u^{T}Su+v^{T}(F+S)v-2u^{T}Sv\\
 & =v^{T}Fv+(u-v)^{T}S(u-v).
\end{align*}
\end{proof}
\begin{lem}
\label{lem:integrang_log_concave}Let $f$ be any function of the
form $f(u):u\in\mathbb{R}^{p}\mapsto\exp(-\frac{1}{2}u^{T}Fu)\tilde{f}(u)$
where $F$ is a real symmetric matrix and $\tilde{f}$ is log-concave,
and let $S$ be a real symmetric matrix such that $F+S$ is invertible.
Then for any $a\in\mathbb{R}^{p}$ and $p\times p$ real matrix $B$,
\[
f(v)\exp\left[-\frac{1}{2}(v-a-Bu)^{T}S(v-a-Bu)\right]=\exp\left(-\frac{1}{2}u^{T}B^{T}CBu\right)\tilde{f}(v)\exp\left[-\frac{1}{2}z^{T}(F+S)z\right],
\]
where $C=F(F+S)^{-1}S$ and $z=v-(F+S)^{-1}S(a+Bu)$
\end{lem}
\begin{proof}
Using lemma \ref{lem:matrix identities} with $u$ there replaced
by $a+Bu$, 
\begin{align*}
 & f(v)\exp\left[-\frac{1}{2}(v-a-Bu)^{T}S(v-a-Bu)\right]\\
 & =\tilde{f}(v)\exp\left[-\frac{1}{2}\left\{ v^{T}Fv+(v-a-Bu)^{T}S(v-a-Bu)\right\} \right]\\
 & =\tilde{f}(v)\exp\left[-\frac{1}{2}\left\{ (a+Bu)^{T}C(a+Bu)+z^{T}(F+S)z\right\} \right]\\
 & =\exp\left(-\frac{1}{2}u^{T}B^{T}CBu\right)\exp\left[-\frac{1}{2}\left(a^{T}Ca+2a^{T}CBu\right)\right]\tilde{f}(v)\exp\left[-\frac{1}{2}z^{T}(F+S)z\right].
\end{align*}
\end{proof}
\begin{proof}
[Proof of proposition \ref{prop:quantitative log-concave}]First note
that for the signal process $(X_{t})_{t\in\mathbb{R}_{+}}$ as per
(\ref{eq:SDE_intro}), 

\begin{align*}
 & m_{t}\coloneqq\mathbf{E}_{x_{0}}[X_{t}]=a_{t}+e^{\beta t}x_{0},\\
 & \Sigma_{t}\coloneqq\mathbf{E}_{x_{0}}[(X_{t}-m_{t})(X_{t}-m_{t})^{T}]=\sigma^{2}\int_{0}^{t}e^{\beta(t-s)}(e^{\beta(t-s)})^{T}\mathrm{d}s,
\end{align*}
where
\[
a_{t}\coloneqq e^{\beta t}\int_{0}^{t}(e^{\beta s})^{-1}\alpha\mathrm{d}s.
\]
It follows that $u^{T}\Sigma_{t}^{-1}u\geq\Lambda_{t}^{-1}u^{T}u$
for all $u\in\mathbb{R}^{p}$ with the shorthand $\Lambda_{t}\coloneqq\sigma^{2}\int_{0}^{t}\lambda_{\beta}^{\mathrm{max}}(s)\mathrm{d}s$. 

Applying lemma \ref{lem:integrang_log_concave} with $a=a_{t}$, $B=e^{\beta t}$,
$S=I\Lambda_{t}^{-1}$, $f=\varphi_{j,k}$, $F=I\lambda_{g}(j,y_{j})$,
and lemma \ref{lem:varphi_log_concave},
\begin{align}
 & \varphi_{j,k}(x)\exp\left[-\frac{1}{2}(x-a_{t}-e^{\beta t}x_{0})^{T}\Sigma_{t}^{-1}(x-a_{t}-e^{\beta t}x_{0})\right]\nonumber \\
 & =\exp\left(-\frac{1}{2}\frac{\lambda_{g}(j,y_{j})\lambda_{\beta}^{\mathrm{min}}(t)}{1+\lambda_{g}(j,y_{j})\Lambda_{t}}x_{0}^{T}x_{0}\right)\nonumber \\
 & \quad\cdot\tilde{\varphi}_{j,k}(x)\exp\left[-\frac{1}{2}z_{t}^{T}z_{t}(\lambda_{g}(j,y_{j})+\Lambda_{t}^{-1})\right]\label{eq:log_concave_decomp1}\\
 & \quad\cdot\exp\left[-\frac{1}{2}\frac{\lambda_{g}(j,y_{j})}{1+\lambda_{g}(j,y_{j})\Lambda_{t}}x_{0}^{T}\left((e^{\beta t})^{T}e^{\beta t}-I\lambda_{\beta}^{\mathrm{min}}(t)\right)x_{0}\right]\label{eq:log_concave_decomp2}\\
 & \quad\cdot\exp\left[-\frac{1}{2}(x-a_{t}-e^{\beta t}x_{0})^{T}(\Sigma_{t}^{-1}-\Lambda_{t}^{-1}I)(x-a_{t}-e^{\beta t}x_{0})\right],\label{eq:log_concave_decomp3}
\end{align}
where $z_{t}=x-(a_{t}+e^{\beta t}x_{0})/(1+\lambda_{g}(j,y_{j})\Lambda_{t})$. 

The product of the terms in (\ref{eq:log_concave_decomp1})-(\ref{eq:log_concave_decomp3})
is jointly log-concave in $(x_{0},x)$. Therefore by lemma \ref{lem:Prekopa},
there exists a function $\tilde{h}$ such that $x\mapsto\tilde{h}(x,t)$
is log-concave and 
\begin{align*}
h(x_{0},t) & =P^{\Delta-t}\varphi_{j,k}(x_{0})\\
 & =\int\varphi_{j,k}(x)\exp\left[-\frac{1}{2}(x-a_{\Delta-t}-e^{\beta(\Delta-t)}x_{0})^{T}\Sigma_{\Delta-t}^{-1}(x-a_{\Delta-t}-e^{\beta(\Delta-t)}x_{0})\right]\mathrm{d}x\\
 & =\exp\left(-\frac{1}{2}\frac{\lambda_{g}(j,y_{j})\lambda_{\beta}^{\mathrm{min}}(\Delta-t)}{1+\lambda_{g}(j,y_{j})\Lambda_{\Delta-t}}x_{0}^{T}x_{0}\right)\tilde{h}(x_{0},t),
\end{align*}
which completes the proof. 
\end{proof}

\subsection{Discussion of theorem \ref{thm:singular}\label{subsec:Discussion-and-examples}}

The aim of this section is to help interpret the quantities on the
right hand side of (\ref{eq:thm_1_bound}) and their combined effect
on the behaviour of (\ref{eq:thm_1_bound}) as $k$ grows.

\subsubsection{Dimension-free nature of the contraction rate}

The quantity $\lambda(j,y,t)$ in (\ref{eq:lambda_thm1}) does not
necessarily depend on the dimension of the state space, $\mathbb{R}^{p}$.
For example, the quantities $\lambda_{g}(j,y)$, $\lambda_{\mathrm{sig}},$
$\lambda_{\beta}^{\mathrm{min}}(t)$, $\lambda_{\beta}^{\mathrm{max}}(t)$
and $\sigma^{2}$ appearing in (\ref{eq:lambda_thm1}) are stable
under tensor products of the model described in section \ref{sec:Introduction},
in the sense that if one expands the model to state-space $\mathbb{R}^{2p}$
by defining the signal to be two independent and identically distributed
copies of (\ref{eq:SDE_intro}), independently observed as $y_{k}=[y_{k}^{(1)}\,y_{k}^{(2)}]\in\mathbb{Y}^{2}$
with likelihood functions having common strong log-concavity parameter
$\lambda_{g}(k,y_{k})$, then there is no degradation of $\lambda(j,y,t)$.
To make this precise note that:

\begin{align*}
 & 1)\;g_{k}(x^{(i)},y_{k}^{(i)})=\exp\left[-\frac{\lambda_{g}(k,y_{k})}{2}\|x^{(i)}\|^{2}\right]\tilde{g}_{k}(x^{(i)},y_{k}^{(i)}),\quad i=1,2,\\
 & \quad\Longrightarrow\quad g_{k}(x^{(1)},y_{k}^{(1)})g_{k}(x^{(2)},y_{k}^{(2)})=\exp\left[-\frac{\lambda_{g}(k,y_{k})}{2}(\|x^{(1)}\|^{2}+\|x^{(2)}\|^{2})\right]\tilde{g}_{k}(x^{(1)},y_{k}^{(1)})\tilde{g}_{k}(x^{(2)},y_{k}^{(2)}),\\
 & 2)\;\mathrm{spectrum}\{(\beta+\beta^{T})/2\}=\mathrm{spectrum}\{(\beta^{\otimes2}+(\beta^{\otimes2})^{T})/2\},\\
 & 3)\;\mathrm{spectrum}\{e^{\beta t}(e^{\beta t})^{T}\}=\mathrm{spectrum}\{e^{\beta^{\otimes2}t}(e^{\beta^{\otimes2}t})^{T}\},
\end{align*}
where $\beta^{\otimes2}$ denotes the Kronecker product $\left[\begin{array}{cc}
1 & 0\\
0 & 1
\end{array}\right]\otimes\beta$. 1) shows that $x=[x^{(1)}\,x^{(2)}]^{T}\mapsto g_{k}(x^{(1)},y_{k}^{(1)})g_{k}(x^{(2)},y_{k}^{(2)})$
is strongly log-concave with parameter $\lambda_{g}(k,y_{k})$. 2)
and 3) imply that $\lambda_{\mathrm{sig}},\lambda_{\beta}^{\mathrm{min}}(t),\lambda_{\beta}^{\mathrm{max}}(t)$
are preserved by expanding the model from $\mathbb{R}^{p}$ to $\mathbb{R}^{2p}$
in the above stated fashion.

\subsubsection{The relationship between signal stability and filter stability}

For the signal model (\ref{eq:SDE_intro}) in general, $\lambda_{\mathrm{sig}}$
could be negative, zero or positive. When $\lambda_{\mathrm{sig}}>0$
the signal is exponentially stable, as follows from:
\begin{lem}
\label{lem:signal_contraction}For any given $\alpha$, $\beta$,
$\sigma$, the transition probabilities $P_{t}(x,\cdot)\coloneqq\mathbf{P}(X_{t}\in\cdot|X_{0}=x)$
of the signal model (\ref{eq:SDE_intro}) satisfy, for any $q\geq1$,
\begin{equation}
W_{q}(P_{t}(x,\cdot),P_{t}(x^{\prime},\cdot))\leq\exp(-\lambda_{\mathrm{sig}}t)\|x-x^{\prime}\|,\qquad\forall x,x^{\prime}\in\mathbb{R}^{p}.\label{eq:signal_contraction}
\end{equation}
\end{lem}
\begin{proof}
[Proof of lemma \ref{lem:signal_contraction}] The proof follows the
same synchronous coupling argument used in the proof of proposition
\ref{prop:wasser_R} but with the $\nabla_{x}\log h$ term there replaced
by zero, so the details are omitted.
\end{proof}
The inequality (\ref{eq:signal_contraction}) cannot be improved in
general. For example, in the case that $p=1$, we have $\beta=-\lambda_{\mathrm{sig}}$,
$P_{t}(x,\cdot)=\mathcal{N}(m_{t}(x),\nu_{t})$ where $\dot{m}_{t}(x)=\alpha+\beta m_{t}(x)$
with $m_{0}(x)=x$ , $\dot{\nu}_{t}=2\beta\nu_{t}+\sigma^{2}$ with
$\nu_{0}=0$, and the order $q=2$ Wasserstein distance is available
in closed form \cite[eq. 3]{dowson1982frechet}:
\begin{equation}
W_{2}(P_{t}(x,\cdot),P_{t}(x^{\prime},\cdot))=|m_{t}(x)-m_{t}(x^{\prime})|=\exp(-\lambda_{\mathrm{sig}}t)|x-x^{\prime}|.\label{eq:signal_contraction_1d}
\end{equation}
Thus when $\lambda_{\mathrm{sig}}\leq0$ the signal is not exponentially
stable in general.

Now let us turn to the question of how $\lambda_{\mathrm{sig}}$ impacts
filter stability. Inspecting (\ref{eq:lambda_thm1}) we observe that
the ratio on the right hand side is always nonnegative, because $\lambda_{g}(k,y)\geq0$
for all $k$ and $y$ under assumption \ref{assu:g}, and $e^{\beta t}(e^{\beta t})^{T}$
is symmetric and positive semidefinite. Therefore with no further
assumptions than those of theorem \ref{thm:singular} the following
bound holds for any $q\geq1$, 
\begin{equation}
W_{q}(\pi_{k}(x,y_{0:k},\cdot),\pi_{k}(x^{\prime},y_{0:k},\cdot))\leq\exp(-\lambda_{\mathrm{sig}}k\Delta)\|x-x^{\prime}\|,\quad\forall x,x^{\prime}\in\mathbb{R}^{p},\,y_{0},\ldots,y_{k}\in\mathbb{Y}.\label{eq:filter_stable_signal}
\end{equation}
Thus when $\lambda_{\mathrm{sig}}>0$, the filter inherits exponential
stability from the signal. The ratio term in (\ref{eq:lambda_thm1})
determines whether or not we can deduce a tighter bound than (\ref{eq:filter_stable_signal})
from (\ref{eq:thm_1_bound}), and in particular determines whether
or not the right hand side of (\ref{eq:thm_1_bound}) converges to
zero as $k\to\infty$ when $\lambda_{\mathrm{sig}}\leq0$.

Introducing a simplifying assumption that $\beta$ is a diagonal matrix
allows us to derive a more easily interpretable upper bound for this
ratio term, which we shall examine in the context of specific observation
models below.
\begin{lem}
\label{lem:beta_diagonal}In addition to assumption \ref{assu:g},
let $\beta$ be a diagonal matrix with maximum and minimum diagonal
elements respectively denoted $\overline{\beta},\underline{\beta}\in\mathbb{R}$.
Then for any $q\geq1$, $k\geq1$ and $y_{0},\ldots,y_{k}\in\mathbb{Y}$,
\begin{multline*}
W_{q}(\pi_{k}(x,y_{0:k},\cdot),\pi_{k}(x^{\prime},y_{0:k},\cdot))\leq\\
\exp\left[k\Delta\overline{\beta}-\sum_{j=0}^{k}e^{-2\Delta(\overline{\beta}-\underline{\beta})}\log\left[1+\frac{\sigma^{2}\lambda_{g}(j,y_{j})}{2\overline{\beta}}\left(e^{2\overline{\beta}\Delta}-1\right)\right]\right]\|x-x^{\prime}\|,\quad\forall x,x^{\prime}\in\mathbb{R}^{p}.
\end{multline*}
\end{lem}
\begin{proof}
Starting from the identity 
\[
\frac{\sigma^{2}\lambda_{g}(j,y)\lambda_{\beta}^{\mathrm{min}}(\Delta-t)}{1+\sigma^{2}\lambda_{g}(j,y)\int_{t}^{\Delta}\lambda_{\beta}^{\mathrm{max}}(\Delta-s)\mathrm{d}s}=\frac{\lambda_{\beta}^{\mathrm{min}}(\Delta-t)}{\lambda_{\beta}^{\mathrm{max}}(\Delta-t)}\left(-\frac{\mathrm{d}}{\mathrm{d}t}\log\left[1+\sigma^{2}\lambda_{g}(j,y)\int_{t}^{\Delta}\lambda_{\beta}^{\mathrm{max}}(\Delta-s)\mathrm{d}s\right]\right),
\]
then integrating by parts and using the fact that under the diagonal
assumption on $\beta$, $\lambda_{\beta}^{\mathrm{min}}(t)=e^{2t\text{\ensuremath{\underbar{\ensuremath{\beta}}}}},\lambda_{\beta}^{\mathrm{max}}(t)=e^{2t\text{\ensuremath{\overline{\beta}}}}$,
\begin{align}
 & \int_{0}^{\Delta}\frac{\sigma^{2}\lambda_{g}(j,y)\lambda_{\beta}^{\mathrm{min}}(\Delta-t)}{1+\sigma^{2}\lambda_{g}(j,y)\int_{t}^{\Delta}\lambda_{\beta}^{\mathrm{max}}(\Delta-s)\mathrm{d}s}\mathrm{d}t\nonumber \\
 & =\frac{\lambda_{\beta}^{\mathrm{min}}(\Delta)}{\lambda_{\beta}^{\mathrm{max}}(\Delta)}\log\left[1+\sigma^{2}\lambda_{g}(j,y)\int_{0}^{\Delta}e^{2(\Delta-s)\text{\ensuremath{\overline{\beta}}}}\mathrm{d}s\right]\label{eq:int_by_parts_1}\\
 & +2(\overline{\beta}-\underline{\beta})\int_{0}^{\Delta}e^{-2(\Delta-t)(\overline{\beta}-\underline{\beta})}\log\left[1+\sigma^{2}\lambda_{g}(j,y)\int_{t}^{\Delta}e^{2(\Delta-s)\text{\ensuremath{\overline{\beta}}}}\mathrm{d}s\right]\mathrm{d}t\label{eq:int_by_parts_2}\\
 & \geq e^{-2\Delta(\overline{\beta}-\underline{\beta})}\log\left[1+\frac{\sigma^{2}\lambda_{g}(j,y)}{2\overline{\beta}}\left(e^{2\overline{\beta}\Delta}-1\right)\right],\label{eq:int_by_parts_lb}
\end{align}
where the lower bound holds by computing the integral on the right
hand side of (\ref{eq:int_by_parts_1}) and using the fact that (\ref{eq:int_by_parts_2})
is nonnegative. The proof is completed by substituting the lower bound
(\ref{eq:int_by_parts_lb}) into the result of theorem \ref{thm:singular}
and noting that under the diagonal assumption on $\beta$, $\lambda_{\mathrm{sig}}=-\overline{\beta}$.
\end{proof}

\subsubsection{Examples\label{subsec:Examples}}

\subsubsection*{Linear-Gaussian observations}

In this case $\mathbb{Y}=\mathbb{R}^{n}$ and for all $k\in\mathbb{N}_{0}$,
\begin{equation}
g_{k}(x,y)=(2\pi)^{-n/2}\det(\Sigma)^{-1/2}\exp\left[-\frac{1}{2}(y-Ax)^{T}\Sigma^{-1}(y-Ax)\right],\label{eq:obs_gauss}
\end{equation}
where $A$ and $\Sigma$ are matrices of appropriate sizes and $\Sigma$
is symmetric and positive definite. For any $u\in\mathbb{R}^{p}$,
$u^{T}A^{T}\Sigma^{-1}Au\geq\|u\|^{2}\lambda_{A^{T}A}^{\mathrm{min}}/\lambda_{\Sigma}^{\mathrm{max}}$
where $\lambda_{A^{T}A}^{\mathrm{min}}$ is the minimum eigenvalue
of $A^{T}A$ and $\lambda_{\Sigma}^{\mathrm{max}}$ is the maximum
eigenvalue of $\Sigma$. Thus for (\ref{eq:obs_gauss}), assumption
\ref{assu:g} holds with $\lambda_{g}(k,y)$ taken to be $\lambda_{A^{T}A}^{\mathrm{min}}/\lambda_{\Sigma}^{\mathrm{max}}$
for all $k$ and $y$.

For ease of exposition, consider the case of diagonal $\beta$ addressed
in lemma \ref{lem:beta_diagonal}. If $\lambda_{A^{T}A}^{\mathrm{min}}=0$,
i.e. $A$ is rank-deficient, and taking $\lambda_{g}(k,y)=\lambda_{A^{T}A}^{\mathrm{min}}/\lambda_{\Sigma}^{\mathrm{max}}$,
the right hand side of the bound in lemma \ref{lem:beta_diagonal}
tends to zero as $k\to\infty$ only if $\lambda_{\mathrm{sig}}=-\overline{\beta}>0$,
i.e. if the signal is stable. On the other hand if again one takes
$\lambda_{g}(k,y)=\lambda_{A^{T}A}^{\mathrm{min}}/\lambda_{\Sigma}^{\mathrm{max}}$,
but now with some fixed $\lambda_{A^{T}A}^{\mathrm{min}}>0$ and $\lambda_{\mathrm{sig}}=-\overline{\beta}\leq0$,
the right hand side of the bound of lemma \ref{lem:beta_diagonal}
tends to zero as $k\to\infty$ if $\sigma^{2}/\lambda_{\Sigma}^{\mathrm{max}}$
is large enough, which means that the level of noise in the observations
is small relative to the level of noise in the signal.

As an example of how the filter can indeed fail to be stable if $\lambda_{A^{T}A}^{\mathrm{min}}\leq0$,
consider the case in which $p=2$, $\beta=\left[\begin{array}{cc}
\beta^{(1)} & 0\\
0 & \beta^{(2)}
\end{array}\right]$ for any $\beta^{(1)},\beta^{(2)}\in\mathbb{R}$, $n=1$ and $A=[0\,1]$,
so the first coordinate of the signal is completely unobserved. In
this scenario it follows from (\ref{eq:filter_defn_k}) that $\pi_{k}(x,y_{0:k},A\times\mathbb{R})=\tilde{P}_{k\Delta}(x^{(1)},A)$,
for any $A\in\mathcal{B}(\mathbb{R})$ , $x=[x^{(1)}\,x^{(2)}]^{T}\in\mathbb{R}^{2}$
and where $(\tilde{P}_{t})_{t\geq0}$ are the transition probabilities
of the first coordinate of the signal process, i.e. the solution of
$\mathrm{d}X_{t}^{(1)}=(\alpha^{(1)}+\beta^{(1)}X_{t}^{(1)})\mathrm{d}t+\sigma\mathrm{d}B_{t}^{(1)}$.
Therefore, using (\ref{eq:signal_contraction_1d}), the filter is
not exponentially stable if $\beta^{(1)}\geq0$.

The pair of conditions that either i) $\lambda_{\mathrm{sig}}>0$,
or that ii) $\lambda_{A^{T}A}^{\mathrm{min}}>0$ and $\sigma^{2}/\lambda_{\Sigma}^{\mathrm{max}}$
is large enough, are together qualitatively similar to the notion
of ``detectability'' in linear systems theory and in terms of which
stability of the Kalman filter can be established, see e.g. \cite{van2010nonlinear}
for a summary and historical references. However it does not seem
easy to make a close comparison to the stability results surveyed
in \cite{van2010nonlinear} because they concern the total variation
distance and involve the observations being random and generated by
the model. By contrast heorem 1 concerns the Wasserstein distance
and subject to assumption 1, the obserations are arbitrary.

\subsubsection*{Stochastic volatility}

In this case $\mathbb{Y}=\mathbb{R}^{p},$ and for all $k\in\mathbb{N}_{0}$,
\[
g_{k}(x,y)=(2\pi)^{-p/2}\det(V(x))^{1/2}\exp\left[-\frac{1}{2}y^{T}V(x)y\right],\qquad V(x)=\mathrm{diag}\{\exp(-x^{(1)}),\cdots,\exp(-x^{(p)})\},
\]
where $x=[x^{(1)}\cdots x^{(p)}]^{T}\in\mathbb{R}^{p}$. Stochastic
volatility models are very popular in econometrics and finance \cite{kim1998stochastic,pitt1999time,jacquier2002bayesian}.
The observations $(y_{k})_{k\in\mathbb{N}_{0}}$, where $y_{k}=[y_{k}^{(1)}\,\cdots\,y_{k}^{(p)}]^{T}$,
represent the returns on a family of $p$ financial assets, whose
time varying volatilities are modelled through the signal process.
Writing out the log-likelihood function:
\[
\log g_{k}(x,y)=-\frac{p}{2}\log2\pi-\frac{1}{2}\sum_{i=1}^{p}x^{(i)}-\frac{1}{2}\sum_{i=1}^{p}(y^{(i)})^{2}\exp(-x^{(i)}),
\]
it is readily checked that assumption \ref{assu:g} is satisfied with
$\lambda_{g}(k,y)=0$, for all $k\in\mathbb{N}_{0}$ and $y\in\mathbb{Y}$.
Therefore for this stochastic volatility model we need to rely on
the condition $\lambda_{\mathrm{sig}}>0$ in order to deduce that
the right hand side of (\ref{eq:thm_1_bound}) tends to zero as $k\to\infty$.
However it is remarkable that such convergence holds without any on
the realized observations $y_{0},\ldots,y_{k}$, compared to analogous
results for stochastic volatility models which concern total variation
rather than Wasserstein distance, e.g \cite[Sec 4.3]{Douc2009}, in
which certain stochastic hypotheses are placed on the observation
sequence in order to prove that the filter forgets its initial condition
almost surely with respect to the law of the observations.

\subsubsection*{Markov random field model for neural data}

In statistical neuroscience, log-concave likelihood functions appear
in Markov random field models used to analyze time-varying correlations
in multivariate neural spike trains \cite{shimazaki2012state,donner2017approximate}.
Here $y_{k}\in\mathbb{Y}=\{0,1\}^{n}$ is a binary vector indicating
the firing pattern $n$ neurons in the $k$th time window, and 
\begin{equation}
g_{k}(x,y)=\exp\left\{ \sum_{i=1}^{n}\sum_{j>i}y^{(i)}y^{(j)}x^{(i,j)}+\sum_{i=1}^{n}y^{(i)}x^{(i)}-\psi(x)\right\} ,\label{eq:obs_model_mrf}
\end{equation}
where $p=n(n-1)+n$, $x$ is a vector with elements $\{x^{(i,j)}$;
$j>i,\,x^{(i)};i=1,\ldots,n\}$ , and $\psi$, called the log-partition
function, is smooth and convex. Assumption \ref{assu:g} holds with
$\lambda_{g}(k,y)=0$ for all $k\in\mathbb{N}_{0}$ and $y\in\mathbb{Y}$.

\subsubsection*{Exponential families and dynamic generalized linear models}

The observation models in (\ref{eq:obs_gauss}) and (\ref{eq:obs_model_mrf}),
as distributions over $y$ parameterized by $x$, are so-called exponential
families of distributions \cite{sundberg2019statistical}. Other exponential
families include the beta, Dirichlet, exponential, Fisher, gamma,
Multinomial, Poisson and Von Mises distributions, to mention just
a few. It is a property of exponential families that their log-likelihood
function, as a function of their canonical parameter, is smooth and
log-concave \cite[Prop 3.10]{sundberg2019statistical}. Thus whenever
$g_{k}(x,y)$ is an exponential family of distributions over $y$
with canonical parameter $x\in\mathbb{R}^{p}$, and $x\mapsto g_{k}(x,y)$
is strictly positive for all $k\in\mathbb{N}_{0}$, $y\in\mathbb{Y}$,
assumption \ref{assu:g} holds with $\lambda_{g}(k,y)=0$ for all
$k\in\mathbb{N}_{0}$ and $y\in\mathbb{Y}$ .

Exponential families of distributions form the building blocks of
Generalized Linear Models \cite{mccullagh1989generalized}. In this
setting $y_{k}=[y_{k}^{(1)}\,\cdots\,y_{k}^{(n)}]^{T}\in\mathbb{Y}\subseteq\mathbb{R}^{n}$
is a vector of response variables whose relationship with covariates
$z_{k}=(z_{k}^{(i,j)})$, $i=1,\ldots,n$, $j=1,\ldots,p$, is modelled
through $g_{k}(x,y_{k})$ of the form:
\[
g_{k}(x,y_{k})=\exp\left[\sum_{i=1}^{n}\left\{ \sum_{j=1}^{p}y_{k}^{(i)}z_{k}^{(i,j)}x^{(j)}-\psi\left(\sum_{j=1}^{p}z_{k}^{(i,j)}x^{(j)}\right)+\log\phi(y_{k}^{(i)})\right\} \right],
\]
where $x=[x^{(1)}\,\cdots\,x^{(p)}]^{T}$ is the vector of regression
parameters, $\phi$ is a given function, and $\psi$ is convex, so
that $x\mapsto g_{k}(x,y_{k})$ is indeed log-concave. The situation
in which the regression parameter $x$ is treated as time-varying
is known as a Dynamic Generalized Linear Model \cite{harrison1999bayesian}.
An example is discussed in section \ref{sec:General-initial-measures}.

\section{Smoothing distributions and a family of weighted Wasserstein distances\label{sec:General-initial-measures}}

It appears to be a nontrivial matter to extend theorem \ref{thm:singular}
to the case where the filter is initialized from two general probability
measures, say $\mu$ and $\nu$ on $\mathcal{B}(\mathbb{R}^{p})$
instead of only at points $x$ and $x^{\prime}$, in a way which can
yield a contractive bound in terms of $W_{q}(\mu,\nu)$. The difficulty
stems from the fact that the generalization of (\ref{eq:pi_theta_R})
to an arbitrary initial distribution $\mu$ is, with a slight overloading
of the notation $\pi_{k}$ in its first argument: 
\begin{equation}
\pi_{k}(\mu,y_{0:k},A)\coloneqq\mu_{0,k}R_{1,k}R_{2,k}\cdots R_{k,k}(A),\qquad\mu_{0,k}(A)\coloneqq\frac{\mu\cdot\varphi_{0,k}}{\mu\varphi_{0,k}},\label{eq:pi_k_mu_defn}
\end{equation}
where the dependence of $\mu_{0,k}$ on $y_{0:k}$ is not shown in
the notation. A direct corollary of theorem \ref{thm:singular} together
with the identity (\ref{eq:pi_theta_R}) is:
\begin{equation}
W_{q}(\pi_{k}(\mu,y_{0:k},\cdot),\pi_{k}(\nu,y_{0:k},\cdot))\leq\exp\left[-\sum_{j=1}^{k}\int_{0}^{\Delta}\lambda(j,y_{j},t)\mathrm{d}t\right]W_{q}(\mu_{0,k},\nu_{0,k}),\label{eq:wasserstein_mu}
\end{equation}
but even if $\lim_{k\to\infty}\exp\left[-\sum_{j=1}^{k}\int_{0}^{\Delta}\lambda(j,y_{j},t)\mathrm{d}t\right]=0$,
it cannot be deduced immediately from (\ref{eq:wasserstein_mu}) that
the left hand side of (\ref{eq:wasserstein_mu}) converges to zero
as $k\to\infty$ due to the dependence of $W_{q}(\mu_{0,k},\nu_{0,k})$
on $k$ and $y_{0:k}$.

An alternative is to work with a certain family of weighted Wasserstein
distances between filtering distributions. As we shall see, this is
equivalent to establishing forgetting of the initial condition for
so-called\emph{ smoothing} distributions, which unlike filtering distributions
condition on future as well as past and present observations. To describe
this equivalence in more detail we shall need the following lemma.
\begin{lem}
Let $d(\cdot,\cdot)$ be a metric on the set of probability measures
on $\mathcal{B}(\mathbb{R}^{p})$ and let $\phi:\mathbb{R}^{p}\to(0,\infty).$
Then $d^{\phi}(\cdot,\cdot)$ defined by:

\[
d^{\phi}\,:\,(\mu,\nu)\longmapsto d\left(\frac{\mu\cdot\phi}{\mu\phi},\frac{\nu\cdot\phi}{\nu\phi}\right)
\]
 is a metric on the subset of probability measures $\{\mu\,\,\mathrm{on}\,\,\mathcal{B}(\mathbb{R}^{p})\,:\,\mu\phi<\infty\}$.
\end{lem}
\begin{proof}
It follows immediately from the assumption that $d$ is a metric and
$\phi$ is strictly positive that on the given domain $\{\mu\,:\,\mu\phi<\infty\}$,
$d^{\phi}$ is nonnegative, symmetric, satisfies the triangle inequality
and $\mu=\nu\Rightarrow d^{\phi}(\mu,\nu)=0$. For the reverse implication,
using the implication $d^{\phi}(\mu,\nu)=0\Rightarrow\mu^{\phi}\coloneqq\frac{\mu\cdot\phi}{\mu\phi}=\frac{\nu\cdot\phi}{\nu\phi}=:\nu^{\phi}$
and the strict positivity of $\phi$ , we have $1=\mathrm{d}\mu^{\phi}/\mathrm{d}\nu^{\phi}=(\mathrm{d}\mu/\mathrm{d}\nu)(\nu\phi/\mu\phi)$,
$\nu$-a.e. Thus $\mathrm{d}\mu/\mathrm{d}\nu$ is a constant $\nu$-a.e.
and since $\mu$ and $\nu$ are probability measures, it follows that
if $d^{\phi}(\mu,\nu)=0$ then $\mu=\nu$.
\end{proof}
Throughout the remainder of section \ref{sec:General-initial-measures}
$(y_{k})_{k\in\mathbb{N}_{0}}$ are an arbitrarily chosen and then
fixed sequence of observations, unless stated otherwise. To avoid
cumbersome formulae, the dependence of some quantities on $(y_{k})_{k\in\mathbb{N}_{0}}$
is not shown in the notation.

Let us introduce the nonnegative integral kernels
\begin{equation}
Q_{k}(x,\mathrm{d}x^{\prime})\coloneqq g_{k-1}(x,y_{k-1})P_{\Delta}(x,\mathrm{d}x^{\prime}).\quad k\geq1.\qquad Q_{j,k}\coloneqq Q_{j+1}\cdots Q_{k},\quad0\leq j<k.\label{eq:Q_defn}
\end{equation}
and the probability measures
\[
\eta_{k}^{\mu}(A)\coloneqq\frac{\mu Q_{0,k}(\mathbf{1}_{A})}{\mu Q_{0,k}(\mathbf{1}_{\mathbb{R}^{p}})},\quad k\geq1,\qquad\eta_{0}^{\mu}\coloneqq\mu,\qquad A\in\mathcal{B}(\mathbb{R}^{p}),
\]
for any $\mu$ such that the denominator is finite. We shall use the
shorthand
\[
\pi_{k}^{\mu}(\cdot)\coloneqq\pi_{k}(\mu,y_{0:k},\cdot).
\]
The dependence of $Q_{k}$ on $y_{k-1}$, of $Q_{j,k}$ on $y_{j},\ldots,y_{k-1}$,
of $\eta_{k}^{\mu}$ on $y_{0},\ldots,y_{k-1}$ and of $\pi_{k}^{\mu}$
on $y_{0},\ldots,y_{k}$ is not shown in the notation. Note from (\ref{eq:pi_k_mu_defn})
that $\eta_{k}^{\mu}(\cdot)=\pi_{k-1}^{\mu}P_{\Delta}(\cdot)$.

We shall use the functions appearing in the following assumption to
define a family of weighted Wasserstein distances.
\begin{assumption}
\label{assu:phi}There exists a probability measure $\mu_{0}$ such
that for the given sequence $(y_{k})_{k\in\mathbb{N}_{0}}$, the following
pointwise limit exists for each $k\in\mathbb{N}_{0}$:
\begin{equation}
\phi_{k,\infty}(x)\coloneqq\lim_{\ell\to\infty}\frac{\varphi_{k,\ell}(x)}{\eta_{k}^{\mu_{0}}\varphi_{k,\ell}},\label{eq:phi_k_infinite_defn-1}
\end{equation}
$\phi_{k,\infty}(x)\in(0,\infty)$ for all $x\in\mathbb{R}^{p}$,
and the functions $(\phi_{k,\infty})_{k\in\mathbb{N}_{0}}$ so-defined
belong to $C^{2}$ and satisfy 
\begin{equation}
Q_{k}\phi_{k,\infty}=\varsigma_{k-1}\phi_{k-1,\infty},\qquad k\geq1,\label{eq:eigen_vector}
\end{equation}
where $\varsigma_{k}\coloneqq\int\eta_{k}^{\mu_{0}}(\mathrm{d}x)g_{k}(x,y_{k})\in(0,\infty)$.
\end{assumption}
Before discussing the interpretation of assumption \ref{assu:phi},
consider the following lemma, which mirrors lemma \ref{lem:varphi_log_concave}.
\begin{lem}
\label{lem:R_infty}If assumption \ref{assu:phi} holds, then for
any $\mu$ such that for all $k\in\mathbb{N}_{0}$, $\pi_{k}^{\mu}P_{\Delta}\phi{}_{k+1,\infty}<\infty$,
the probability measures $(\pi_{k,\infty}^{\mu})_{k\in\mathbb{N}_{0}}$
defined by:
\begin{equation}
\pi_{k,\infty}^{\mu}(A)\coloneqq\frac{\pi_{k}^{\mu}(\mathbf{1}_{A}P_{\Delta}\phi{}_{k+1,\infty})}{\pi_{k}^{\mu}P_{\Delta}\phi{}_{k+1,\infty}},\quad A\in\mathcal{B}(\mathbb{R}^{p}),\label{eq:pi_k_inf_defn}
\end{equation}
satisfy 
\begin{equation}
\pi_{k,\infty}^{\mu}(A)=\pi_{0,\infty}^{\mu}R_{1,\infty}\cdots R_{k,\infty}(A),\label{eq:pi_s_pi}
\end{equation}
with the Markov kernels
\[
R_{k,\infty}(x,\mathrm{d}x^{\prime})\coloneqq\frac{P_{\Delta}(x,\mathrm{d}x^{\prime})\phi{}_{k,\infty}(x^{\prime})}{P_{\Delta}\phi{}_{k,\infty}(x)}.
\]
If additionally assumption \ref{assu:g} holds, then for each $k\in\mathbb{N}_{0},$
there exists a log-concave function $\tilde{\phi}_{k,\infty}$ such
that
\[
\phi{}_{k,\infty}(x)=\exp\left[-\frac{\lambda_{g}(k,y_{k})}{2}\|x\|^{2}\right]\tilde{\phi}_{k,\infty}(x).
\]
\end{lem}
\begin{proof}
To establish (\ref{eq:pi_s_pi}) it suffices to show $\pi_{k-1,\infty}^{\mu}R_{k,\infty}=\pi_{k,\infty}^{\mu}$.
We have
\begin{align*}
\pi_{k-1,\infty}^{\mu}R_{k,\infty}(A) & =\frac{\pi_{k-1}^{\mu}(P_{\Delta}(\phi_{k,\infty})R_{k,\infty}(\mathbf{1}_{A}))}{\pi_{k-1}^{\mu}P_{\Delta}{}_{k}\phi_{k,\infty}}\\
 & =\frac{\pi_{k-1}^{\mu}P_{\Delta}(\mathbf{1}_{A}\phi_{k,\infty})}{\pi_{k-1}^{\mu}P_{\Delta}\phi_{k,\infty}}\\
 & =\frac{\pi_{k-1}^{\mu}P_{\Delta}(\mathbf{1}_{A}Q_{k+1}\phi{}_{k+1,\infty})}{\pi_{k-1}^{\mu}P_{\Delta}Q_{k+1}\phi{}_{k+1,\infty}}\\
 & =\frac{\pi_{k}^{\mu}(\mathbf{1}_{A}P_{\Delta}\phi{}_{k+1,\infty})}{\pi_{k}^{\mu}P_{\Delta}\phi{}_{k+1,\infty}}=\pi_{k,\infty}^{\mu}(A),
\end{align*}
where (\ref{eq:eigen_vector}), (\ref{eq:Q_defn}) and the identity
$\pi_{k}^{\mu}(A)=\pi_{k-1}^{\mu}[P_{\Delta}(\mathbf{1}_{A}Q_{k}(\mathbf{1}_{\mathbb{R}^{p}}))]/\pi_{k-1}^{\mu}[P_{\Delta}(Q_{k}(\mathbf{1}_{\mathbb{R}^{p}}))]$
have been used. 

For the second claim, the fact that $\phi{}_{j,\infty}$ is log-concave
for every $j\in\mathbb{N}_{0}$ follows from its definition as the
pointwise limit in (\ref{eq:phi_k_infinite_defn-1}) and the log-concavity
of $\varphi_{j,k}$ established in lemma \ref{lem:varphi_log_concave}.
By lemma \ref{lem:preservation_basic}, $P_{\Delta}\phi{}_{k+1}$
is log-concave and since by assumption \ref{assu:phi}, $\phi{}_{k,\infty}=\varsigma_{k}^{-1}Q_{k+1}\phi{}_{k+1,\infty}$,
we may take $\tilde{\phi}_{k,\infty}(x)=\varsigma_{k}^{-1}\tilde{g}_{k}(x,y_{k})P_{\Delta}\phi{}_{k+1,\infty}(x)$. 
\end{proof}
Since $\pi_{k}^{\mu}$ has the interpretation of the conditional distribution
of $x_{k\Delta}$ given $(y_{0},\ldots,y_{k})$, the measure $\pi_{k}^{\mu}\cdot(P_{\Delta}\varphi_{k+1,\ell})/\pi_{k}^{\mu}P_{\Delta}\varphi_{k+1,\ell}$
is the so-called smoothing distribution which conditions additionally
on $(y_{k+1},\cdots,y_{k+\ell})$. The interpretation of (\ref{eq:phi_k_infinite_defn-1})
is then that $\phi_{k,\infty}$ is the function with which to re-weight
$\pi_{k}^{\mu}P_{\Delta}$ in order to condition on the infinite data
record $(y_{k+\ell})_{\ell\in\mathbb{N}_{0}}$. Indeed it is clear
from (\ref{eq:pi_k_inf_defn}) that assumption \ref{assu:phi} implies
that the filtering and smoothing measures, $\pi_{k}^{\mu}$ and $\pi_{k,\infty}^{\mu}$,
are equivalent, despite the fact that $\pi_{k,\infty}^{\mu}$ conditions
on an infinite number of observations.

The question of whether there exists a function which achieves this
conditioning is itself closely connected to the question of filter
stability. For a general class of discrete-time filtering problems
with an ergodic signal and nondegenerate observations, it is shown
in \cite[Lemma 3.8]{van2009stability} (see also the commentary immediately
after the proof of Lemma 3.6 in the same article), that the transition
kernel of the signal conditional on an infinite future data record
is absolutely continuous w.r.t. to the (unconditional) transition
kernel of the signal. In the notation of the present work this is,
for each $k$, the absolute continuity of $R_{k,\infty}(x,\cdot)$
w.r.t. $P_{\Delta}(x,\cdot)$, i.e. $\phi_{k,\infty}$ is (a version
of) the corresponding Radon-Nikodym derivative up to a factor depending
on $x$. See \cite{whiteley2014twisted} for a discussion on doubly
infinite time horizons but under much more restrictive conditions.
Assumption 2 requires such a derivative to not only exist but also
satisfy certain regularity conditions, which below shall be verified
in the setting of a specific example using the techniques of \cite{whiteley2013stability}.
It is an open question whether assumption \ref{assu:phi} can be deduced
directly from theorem \ref{thm:singular}.

When assumption \ref{assu:phi} holds, we shall consider the family
of weighted Wasserstein distances
\[
W_{q,k}(\mu,\nu)\coloneqq W_{q}\left(\frac{\mu\cdot P_{\Delta}\phi{}_{k+1,\infty}}{\mu P_{\Delta}\phi{}_{k+1,\infty}},\frac{\nu\cdot P_{\Delta}\phi{}_{k+1,\infty}}{\nu P_{\Delta}\phi{}_{k+1,\infty}}\right),\qquad k\in\mathbb{N}_{0},
\]
whenever $\mu,\nu$ satisfy appropriate integrability conditions for
these distances to be well-defined. The interest in the distances
$W_{q,k}$ is due to the identity:
\begin{equation}
W_{q,k}(\pi_{k}^{\mu},\pi_{k}^{\nu})=W_{q}(\pi_{k,\infty}^{\mu},\pi_{k,\infty}^{\nu}),\label{eq:wass_id}
\end{equation}
which follows from (\ref{eq:pi_k_inf_defn}). Thus $W_{q,k}$ quantifies
distance between $\pi_{k}^{\mu}$ and $\pi_{k}^{\nu}$ as the $W_{q}$-distance
between the corresponding smoothing distributions $\pi_{k,\infty}^{\mu}$
and $\pi_{k,\infty}^{\nu}$.

We denote the set of probability measures
\[
\mathcal{P}_{q}\coloneqq\left\{ \mu\,\,\mathrm{on}\,\,\mathcal{B}(\mathbb{R}^{p})\,:\,\int(1+\|u\|^{q})\phi{}_{0}(u)\mu(\mathrm{d}u)<\infty\quad\mathrm{and}\quad\pi_{k}^{\mu}P_{\Delta}\phi{}_{k+1,\infty}<\infty,\;\forall k\in\mathbb{N}_{0}\right\} .
\]

\begin{thm}
\label{thm:general}If assumption \ref{assu:g} holds and for a given
observation sequence $(y_{k})_{k\in\mathbb{N}_{0}}$ assumption\ref{assu:phi}
holds, then for any $q\geq1$,
\[
W_{q,k}(\pi_{k}(\mu,y_{0:k},\cdot),\pi_{k}(\nu,y_{0:k},\cdot))\leq\exp\left[-\sum_{j=1}^{k}\int_{0}^{\Delta}\lambda(j,y_{j},t)\mathrm{d}t\right]W_{q,0}(\pi_{0}^{\mu},\pi_{0}^{\nu}),\qquad\forall k\geq1,\mu,\nu\in\mathcal{P}_{q},
\]
 where $\lambda(j,y_{j},t)$ is as in theorem \ref{thm:singular}. 
\end{thm}
Given the identities (\ref{eq:pi_s_pi}) and (\ref{eq:wass_id}),
the proof of theorem \ref{thm:general} follows almost exactly the
same programme as the proof of theorem \ref{thm:singular}, except
working with the kernels $R_{k,\infty}$, the functions $\phi{}_{k,\infty}$
and their log-concavity in lemma \ref{lem:R_infty}, instead of $R_{j,k}$,
$\varphi_{j,k}$ and their log-concavity in lemma \ref{lem:varphi_log_concave}.
Therefore the details are omitted. The requirement $\mu,\nu\in\mathcal{P}_{q}$
ensures that $W_{q,0}(\mu,\nu)$ and $\pi_{k,\infty}^{\mu},\pi_{k,\infty}^{\nu}$
are well-defined. 

\subsection*{Example: dynamic logistic regression}

As an example of the dynamic Generalized Linear Models described in
section \ref{subsec:Examples}, consider the case: $\sigma^{2}>0$,
$\beta$ such that $\lambda_{\mathrm{sig}}>0$, and with $\mathbb{Y}=\{0,1\}^{n}$,
the observations $Y_{k}=[Y_{k}^{(1)}\,\cdots\,Y_{k}^{(n)}]^{T}$ are
conditionally independent given $x_{k\Delta}$, with the conditional
probability of $\{Y_{k}^{i}=1\}$ being $1/(1+e^{-\sum_{j}x_{k\Delta}^{(j)}z_{k}^{(i,j)}})$,
where $z_{k}^{(i,j)}$ are known covariates. The likelihood function
at time $k$ is then: 
\[
g_{k}(x,y_{k})=\exp\left[\sum_{i=1}^{n}\left\{ \sum_{j=1}^{p}y_{k}^{(i)}z_{k}^{(i,j)}x^{(j)}-\log\left(1+e^{\sum_{j=1}^{p}z_{k}^{(i,j)}x^{(j)}}\right)\right\} \right].
\]

For any $(y_{k})_{k\in\mathbb{N}_{0}}$, assumption \ref{assu:g}
is satisfied with $\lambda_{g}(k,y_{k})=0$, and therefore (\ref{eq:filter_stable_signal})
holds by theorem \ref{thm:singular}. Checking assumption \ref{assu:phi}
is more involved, we shall use some results from \cite{whiteley2013stability}. 

Let us assume that the covariates satisfy
\begin{equation}
\sup_{k\geq0}\sum_{i,j}(z_{k}^{(i,j)})^{2}<\infty,\label{eq:covariates}
\end{equation}
and fix an arbitrarily sequence of observations $(y_{k})_{k\in\mathbb{N}_{0}}$.

The following properties of this model are easily checked (see \cite[Sec. 3.1]{whiteley2013stability}
for a similar example): there exists a constant $c>0$ such that with
\begin{equation}
V(x)\coloneqq1+c\|x\|,\qquad C_{d}\coloneqq\{x\in\mathbb{R}^{p}:V(x)\leq d\},\label{eq:V_defn}
\end{equation}
we have for some $\underline{d}\in[1,\infty)$ and all $d\geq\underline{d}$,
\begin{itemize}
\item $\sup_{k}g_{k}(x,y_{k})\leq1$, $\forall x\in\mathbb{R}^{p}$, and
there exist constants $\delta\in(0,1)$, $b_{d}\in[0,\infty)$ such
that 
\begin{equation}
P_{\Delta}(e^{V})\leq\exp(V(1-\delta)+b_{d}\mathbf{1}_{C_{d}}),\label{eq:drift}
\end{equation}
\item $\inf_{k}g_{k}(x,y_{k})P_{\Delta}(x,C_{d})>0$, $\forall x\in\mathbb{R}^{p}$,
\item there exist constants $\epsilon_{d}^{-},\epsilon_{d}^{+}$ such that
$\forall x\in C_{d}$ and $k\in\mathbb{N}_{0}$,
\[
\epsilon_{d}^{-}\nu_{d}(\mathrm{d}x^{\prime})\mathbf{1}_{C_{d}}(x^{\prime})\leq g_{k}(x,y_{k})P_{\Delta}(x,\mathrm{d}x^{\prime})\mathbf{1}_{C_{d}}(x^{\prime})\leq\epsilon_{d}^{-}\nu_{d}(\mathrm{d}x^{\prime})\mathbf{1}_{C_{d}}(x^{\prime}),
\]
where the probability measure $\nu_{d}$ is the normalized restriction
of Lebesgue measure to $C_{d}$.
\end{itemize}
Define the norm on functions $f:\mathbb{R}^{p}\to\mathbb{R}$, $\|f\|_{e^{V}}\coloneqq\sup_{x}|f(x)|/e^{V(x)}$.
\begin{prop}
\label{prop:L_v}For any $\mu_{0}$ such that $\mu_{0}(e^{V})<\infty$,
define $\phi{}_{j,k}(x)\coloneqq\varphi_{j,k}(x)/\pi_{j-1}^{\mu_{0}}P_{\Delta}\varphi_{j,k}$
. Then:

1) $\sup_{k\geq0}\eta_{k}^{\mu_{0}}(e^{V})<\infty$

2) $\sup_{0\leq j\leq k}\|\phi{}_{j,k}\|_{e^{V}}<\infty$,

3) for all $d\geq\underline{d}$, $\inf_{0\leq j\leq k}\inf_{x\in C_{d}}\phi_{j,k}(x)>0$,

4) for all $0<j\leq k$ , $Q_{j}\phi{}_{j,k}=\varsigma_{j-1}\phi{}_{j-1,k}$,
where $\varsigma_{j}=\int\eta_{j}^{\mu_{0}}(\mathrm{d}x)g_{j}(x,y_{j})$,

5) there exist constants $\rho<1$ and $c_{\mu_{0}}<\infty$ such
that for any $f:\mathbb{R}^{p}\to\mathbb{R}$ with $\|f\|_{e^{V}}<\infty$,
\[
\left|\frac{Q_{j,k}f(x)}{\prod_{i=j}^{k-1}\varsigma_{i}}-\phi{}_{j,k-1}(x)\eta_{k}^{\mu_{0}}f\right|\leq\rho^{k-j}\|f\|_{e^{V}}c_{\mu_{0}}e^{V(x)}\mu_{0}(e^{V}),\quad\forall x\in\mathbb{R}^{p},\,0\leq j<k
\]
\end{prop}
\begin{proof}
The properties identified immediately before the statement of proposition
and the requirement $\mu_{0}(e^{V})<\infty$ imply that conditions
(H1)-(H4) of \cite{whiteley2013stability} are satisfied. Then 1)
and 2) are established by \cite[Prop. 1 and 2]{whiteley2013stability},
3) by \cite[Lem. 10]{whiteley2013stability}, 4) by \cite[Lem.1]{whiteley2013stability},
and 5) by \cite[Thm. 1]{whiteley2013stability}. 

The following proposition establishes that the conditions of theorem
\ref{thm:general} are satisfied. 
\end{proof}
\begin{prop}
\label{prop:wass_logit}For any sequence of observations $(y_{k})_{k\in\mathbb{N}_{0}}$,
the dynamic logistic regression model described above satisfies assumption
\ref{assu:phi} with $\sup_{k\geq0}\|\phi_{k,\infty}\|_{e^{V}}<\infty$,
and for any $q\geq1$,
\begin{equation}
W_{q,k}(\pi_{k}(\mu,y_{0:k},\cdot),\pi_{k}(\nu,y_{0:k},\cdot))\leq\exp\left(-k\Delta\lambda_{\mathrm{sig}}\right)W_{q,0}(\pi_{0}^{\mu},\pi_{0}^{\nu}),\label{eq:wass_bound_logit}
\end{equation}
 for all $\mu,\nu$ in the set of probability measures $\left\{ \mu\;\mathrm{on}\;\mathcal{B}(\mathbb{R}^{p})\;:\;\int(1+\|x\|^{q})e^{c\|x\|}\mu(\mathrm{d}x)<\infty\right\} $
where $c$ is as in (\ref{eq:V_defn}).
\end{prop}
\begin{rem}
The constant $\rho<1$ appearing in part 5) of proposition \ref{prop:L_v}
and obtained using the techniques of \cite{whiteley2013stability}
may degrade with dimension of the state-space. Note however, that
$\rho$ does not appear in (\ref{eq:wass_bound_logit}), it only serves
as an intermediate tool used to in the following proof to help establish
that assumption \ref{assu:phi} holds.
\end{rem}
\begin{proof}
[Proof of proposition \ref{prop:wass_logit}]Choose any $\mu_{0}$
such that $\mu_{0}(e^{V})<\infty$. Noting the identities $\pi_{k-1}^{\mu_{0}}P_{\Delta}\varphi_{k,\ell}=\prod_{j=k}^{\ell}\varsigma_{j}$
and $\phi{}_{j,k}=Q_{j,k+1}\mathbf{1}_{\mathbb{R}^{p}}/\prod_{i=j}^{k}\varsigma_{i}$,
we have for any $\ell\geq1$,
\[
\phi{}_{j,k}-\phi{}_{j,k+\ell}=\frac{Q_{j,k+1}}{\prod_{i=j}^{k}\varsigma_{i}}\left(1-\frac{Q_{k+1,k+\ell+1}\mathbf{1}_{\mathbb{R}^{p}}}{\prod_{i=k+1}^{k+\ell}\varsigma_{i}}\right).
\]
Since $\prod_{i=k+1}^{k+\ell}\varsigma_{i}=\eta_{k+1}^{\mu_{0}}Q_{k+1,k+\ell+1}\mathbf{1}_{\mathbb{R}^{p}}$,
we have $\eta_{k+1}^{\mu_{0}}(1-\frac{Q_{k+1,k+\ell+1}\mathbf{1}_{\mathbb{R}^{p}}}{\prod_{i=k+1}^{k+\ell}\varsigma_{i}})=0$
and by part 2) of proposition \ref{prop:L_v}, $\sup_{j,k,\ell}\frac{\|Q_{k+1,k+\ell+1}\mathbf{1}_{\mathbb{R}^{p}}\|_{e^{V}}}{\prod_{i=k+1}^{k+\ell}\varsigma_{i}}=:c_{Q}<\infty$,
so an application of part 5) of proposition \ref{prop:L_v} gives:
\[
\|\phi{}_{j,k}-\phi{}_{j,k+\ell}\|_{e^{V}}\leq\rho^{k+1-j}c_{Q}c_{\mu_{0}}\mu_{0}(e^{V}),\quad\forall\ell\geq1.
\]
It follows for each $j$, $(\phi{}_{j,k})_{k\geq j}$ is a Cauchy
sequence in the Banach space of functions $f:\mathbb{R}^{p}\rightarrow\mathbb{R}$
endowed with the norm $\|f\|_{e^{V}}<+\infty$. With the strong limit
of $(\phi{}_{j,k})_{k\geq j}$ then denoted $\phi{}_{j,\infty},$
we have $\|\phi{}_{j,\infty}\|_{e^{V}}<\infty$ and $\phi{}_{j,\infty}(x)=\lim_{k\to\infty}\phi{}_{j,k}(x)$
pointwise. 

From part 4) of proposition \ref{prop:L_v}, 
\[
Q_{j}\phi{}_{j,k}=Q_{j}\phi{}_{j,\infty}+Q_{j}(\phi{}_{j,k}-\phi{}_{j,\infty})=\varsigma_{j-1}\phi{}_{j-1,\infty}+\varsigma_{j-1}(\phi{}_{j-1,k}-\phi{}_{j-1,\infty})=\varsigma_{j-1}\phi{}_{j-1,k},
\]
and since using (\ref{eq:drift}), $\|Q_{j}(e^{V})\|_{e^{V}}<\infty$,
$\|\phi{}_{j-1,k}-\phi{}_{j-1,\infty}\|\to0$ and $\|Q_{j}(\phi{}_{j,k}-\phi{}_{j,\infty})\|_{e^{V}}\leq\|Q_{j}(e^{V})\|_{e^{V}}\|\phi{}_{j,k}-\phi{}_{j,\infty}\|_{e^{V}}\to0$,
both as $k\to\infty$, we have $Q_{j}\phi{}_{j,\infty}=\varsigma_{j-1}\phi{}_{j-1,\infty}$.
Since $g_{j}(x,y_{j})\in(0,1)$, we have $\varsigma_{j}\in(0,1)$
and using part 3) of proposition \ref{prop:L_v}, $Q_{j}\phi{}_{j,\infty}(x)>0$
for all $x$ hence $\phi{}_{j-1,\infty}(x)>0$ for all $x$. Also
$\|\phi{}_{j,\infty}\|_{e^{V}}<\infty$ implies $\phi{}_{j,\infty}(x)<\infty$
for all $x$. The membership $\phi{}_{j-1,\infty}\in C^{2}$ follows
from $Q_{j}\phi{}_{j,\infty}=\varsigma_{j-1}\phi{}_{j-1,\infty}$
together with $x\mapsto g_{j-1}(x,y_{j-1})\in C^{2}$ by assumption
\ref{assu:g} and the fact that $P_{\Delta}(x,\cdot)$ is Gaussian
with mean depending linearly on $x$. That completes the verification
of assumption \ref{assu:phi}. 

To complete the proof, observe that in order for $\mu\in\mathcal{P}_{q}$
it is sufficient that $\int(1+\|x\|^{q})e^{V(x)}\mu(\mathrm{d}x)<\infty$,
because using part 2) of proposition \ref{prop:L_v} , $\sup_{k\geq0}\|\phi{}_{k,\infty}\|_{e^{V}}<\infty$,
we have $\pi_{k-1}^{\mu}P_{\Delta}=\eta_{k}^{\mu}$ and by part 1)
of proposition \ref{prop:L_v}, $\sup_{k}\eta_{k}^{\mu}(e^{V})<\infty$.
\end{proof}
\begin{acknowledgement*}
The author thanks Anthony Lee for helpful comments.
\end{acknowledgement*}
\bibliographystyle{plain}
\bibliography{filtering}

\end{document}